\documentclass{amsart}

\usepackage{amsfonts}
\usepackage{amsmath}
\usepackage{hyperref}
\usepackage{amssymb}

\usepackage{amsfonts}
\usepackage{graphicx}
\usepackage{amscd}

\newtheorem{theorem}{Theorem}
\theoremstyle{plain}

\newtheorem{corollary}{Corollary}

\newtheorem{definition}{Definition}

\newtheorem{lemma}{Lemma}

\newtheorem{remark}{Remark}

\numberwithin{equation}{section}

\begin{document}

\title[The Hermite--Hadamard inequality on hypercuboid] {The Hermite--Hadamard inequality on hypercuboid}

\author[M.W. Alomari]{Mohammad W. Alomari}

\address{Department of Mathematics,
Faculty of Science and Information Technology, Irbid National
University, 21110 Irbid, Jordan.} \email{mwomath@gmail.com}

\date{\today}


\begin{abstract}
Given any ${\bf{a}}: = \left( {a_1 ,a_2 , \ldots ,a_n } \right)$
and ${\bf{b}}: = \left( {b_1 ,b_2 , \ldots ,b_n } \right)$ in
$\mathbb{R}^n$. The $\textbf{n}$-fold convex function defined on
$\left[ {{\bf{a}},{\bf{b}}} \right]$, ${\bf{a}},{\bf{b}} \in
\mathbb{R}^n$ with ${\bf{a}}<{\bf{b}}$ is a convex function in
each variable separately. In this work we prove an inequality of
Hermite-Hadamard type for $\textbf{n}$-fold convex functions.
Namely, we establish the inequality
\begin{align*} f\left( {\frac{{{\bf{a}} + {\bf{b}}}}{2}} \right)
\le \frac{1}{{{\bf{b}} - {\bf{a}}}}\int_{\bf{a}}^{\bf{b}} {f\left(
{\bf{x}} \right)d{\bf{x}}}  \le \frac{1}{{2^n
}}\sum\limits_{\bf{c}} {f\left( {\bf{c}} \right)},
\end{align*}
where $\sum\limits_{\bf{c}} {f\left( {\bf{c}} \right)} : =
\sum\limits_{\mathop {c_i  \in \left\{ {a_i ,b_i }
\right\}}\limits_{1 \le i \le n} } {f\left( {c_1, c_2, \ldots ,c_n
} \right)}$. Some other related result are given.
\end{abstract}
\maketitle

\section{Introduction}

The classical Hermite-Hadamard inequality
\begin{align}
\label{HH}f\left( {\frac{{a + b}}{2}} \right) \le \frac{1}{{b -
a}}\int_a^b {f\left( t \right)dt}  \le \frac{{f\left( a \right) +
f\left( b \right)}}{2}
\end{align}
holds for all convex functions defined on a real interval
$[a,b]$.\\

Along the past thirty years, several authors give an attention for
various kind of this inequality and related type inequalities.
Indeed the history of (\ref{HH}) is very long to summarize in one
or two paragraph, however, we can simply say without any worry,
the real work over all these thirty years started in 1992 by
Dragomir \cite{Dragomir1}. In literature, the referenced work
\cite{Dragomir1} was considered as base to study and investigate
(\ref{HH}) by many other authors later.

A progressive work make many interested authors to generalize
(\ref{HH}) and establish a number of formulation in various forms.
In sequence of papers, Dragomir proved various inequalities of
Hermite-Hadamard type for several assumption for the functions
involved; e.g., convex mappings defined on a disk in the plane and
convex mappings defined on a ball in the space . For a
comprehensive work regarding (\ref{HH}) the reader may refer to
\cite{Dragomir1}.

In 2006, de la Cal and C\'{a}rcamo \cite{Cal} studied the
Hermite-Hadamard type for convex functions on $n$-dimensional
convex bodies by translating the problem into of finding
appropriate majorants of the involved random vector for the usual
convex order. Two main results was obtained in \cite{Cal} the
first one regarding mappings defined on polytopes in
$\mathbb{R}^n$, while the second result discussed (\ref{HH}) for
symmetric random vectors taking values in a closed ball for a
given (but arbitrary) norm on $\mathbb{R}^n$, (see also
\cite{Cal1}). In 2008,  a formulation on simplicies was presented;
the key idea of the presented approach was passed through a volume
type formula and its higher dimensional generalization. In 2009,
by using of a stochastic approach, de la Cal \emph{et. al.}
established a multidimensional version of the classical
Hermite-Hadamard inequalities which holds for convex functions on
general convex bodies. In 2012, Yang \cite{Yang} proved an
extension of (\ref{HH}) for functions defined on a convex subsets
of $\mathbb{R}^3$, indeed the author introduced a version of
(\ref{HH}) for function $f$ defined on an annulus domain.
Recently, Moslehian \cite{Moslehian} introduced several matrix and
operator inequalities of Hermite–-Hadamard type and presented some
operator inequalities of Hermite-Hadamard type in which the
classical convexity was used instead of the operator convexity.
For more details, generalization and counterparts the reader may
refer to
\cite{alomari}--\cite{Yang} and the references therein.\\

Let us consider the bi-dimensional interval $  \Delta : = \left[
{a,b} \right] \times \left[ {c,d} \right]  $ in $ \mathbb{R}^2 $
with $a < b$ and $c < d$. Recall that the mapping $ f:\Delta \to
\mathbb{R} $ is convex in $\Delta$ if
\begin{eqnarray*}
f\left( {\lambda x + \left( {1 - \lambda } \right)z,\lambda y +
\left( {1 - \lambda } \right)w} \right) \le \lambda f\left( {x,y}
\right) + \left( {1 - \lambda } \right)f\left( {z,w} \right)
\end{eqnarray*}
holds for all $(x,y), (z, w)$ $\in \Delta$ and $\lambda \in \left[
{0,1} \right]$.\\

Dragomir \cite{Dragomir2} established a new concept of convexity
which is called the co--ordinated convex function, as follows:

A function $ f:\Delta \to \mathbb{R} $ is convex in $\Delta$ is
called co--ordinated convex on $\Delta$ if the partial mappings $
f_y :\left[ {a,b} \right] \to \mathbb{R} $, $ f_y \left( u \right)
= f\left( {u,y} \right)$ and $ f_x :\left[ {c,d} \right] \to
\mathbb{R} $, $ f_x \left( v \right) = f\left( {x,v} \right)$, are
convex for all $ y \in \left[ {c,d} \right] $ and $ x \in \left[
{a,b} \right] $.\\

In \cite{Dragomir2}, Dragomir established the following similar
inequality of Hadamard's type for co--ordinated convex mapping on
a rectangle from the plane $\mathbb{R}^2 $.
\begin{theorem}\label{thm1}
Suppose that $ f:\Delta \to \mathbb{R} $ is co--ordinated convex
on $\Delta$. Then one has the inequalities
\begin{eqnarray}
\label{eq2}f\left( {\frac{{a + b}}{2},\frac{{c + d}}{2}}
\right)&\le& \frac{1}{{\left( {b - a} \right)\left( {d - c}
\right)}}\int\limits_a^b {\int\limits_c^d {f\left( {x,y}
\right)dydx} }\nonumber\\
&\le& \frac{{f\left( {a,c} \right) + f\left( {a,d} \right) +
f\left( {b,c} \right) + f\left( {b,d} \right)}}{4}
\end{eqnarray}
The above inequalities are sharp.
\end{theorem}
In \cite{alomari}, Alomari proved the weighted version of
(\ref{eq2}) which is known as Fej\'{e}r inequality, as follows:
\begin{theorem}
Let $f: [a,b] \times [c,d] \to \mathbb{R}$ be a co--ordinated
convex function, Then the double inequality
\begin{eqnarray}
f\left( {\frac{{a + b}}{2},\frac{{c + d}}{2}} \right) &\le&
\frac{{\int\limits_a^b {\int\limits_c^d {f\left( {x,y}
\right)p\left( {x,y} \right)dydx} } }}{{\int\limits_a^b
{\int\limits_c^d {p\left( {x,y} \right)dydx} } }}
\\
&\le& \frac{{f\left( {a,c} \right) + f\left( {c,d} \right) +
f\left( {b,c} \right) + f\left( {b,d} \right)}}{4}\nonumber
\end{eqnarray}
holds, where $p:\left[ {a,b} \right] \times \left[ {c,d} \right]
\to \mathbb{R}$ is positive, integrable, and symmetric about $x =
\frac{{a + b}}{2}$ and $y = \frac{{c + d}}{2}$. The above
inequalities are sharp.
\end{theorem}

In this work, a new inequality of Hermite-Hadamard type on
hypercuboid is proved.

\section{${\bf{n}}$-fold convex functions}
Given any ${\bf{a}}: = \left( {a_1 ,a_2 , \ldots ,a_n } \right)$
and ${\bf{b}}: = \left( {b_1 ,b_2 , \ldots ,b_n } \right)$ in
$\mathbb{R}^n$, we define
\begin{equation*}
{\bf{a}} \le {\bf{b}} \Longleftrightarrow a_i  \le b_i ,\forall
i,\,\,1 \le i \le n.
\end{equation*}
Clearly, this is a partial order on $\mathbb{R}^n$, and it may be
called the $\mathbf{product \,\,order}$ or the
$\mathbf{componentwise \,\,order}$ on $\mathbb{R}^n$. If $n > 1$,
then the product order on $\mathbb{R}^n$ is not a total order; for
example, if ${\bf{x}} := (1, 0, 0, \cdots , 0)$ and ${\bf{y}} :=
(0, 1, 0, \cdots , 0)$, then neither $x \le y$ nor $y \le x$.

Let
$$I_{{\bf{a}},{\bf{b}}} : = \prod\limits_{i = 1}^n {I_{a_i ,b_i }
}  = I_{a_1 ,b_1 }  \times \cdots  \times I_{a_n ,b_n }.$$

A subset $I$ of $\mathbb{R}^n$ is said to be an $n$-interval if
$I_{{\bf{a}},{\bf{b}}} \subseteq I$ for every ${\bf{a}},{\bf{b}}
\in I$. For example, if $I_1, \ldots , I_n$ are intervals in
$\mathbb{R}$, then $I_1  \times  \cdots  \times I_n$  is an
$n$-interval. Furthermore, an $n$-interval of the form $I_1 \times
\cdots  \times I_n$, where each of the $I_1  \times  \cdots \times
I_n$ is a closed and bounded interval in $\mathbb{R}$, is called a
$\textbf{hypercuboid}$ in $\mathbb{R}^n$.

Throughout this paper, we will consider, for all $a_i ,b_i  \in
\mathbb{R}$, $\left[ {{\bf{a}},{\bf{b}}} \right]: =
\prod\limits_{1 \le i \le n} {\left[ {a_i ,b_i } \right]}$, and
${\bf{c}} = \left( {c_1 ,c_2 , \ldots ,c_n } \right)$, $c_i  \in
\left\{ {a_i ,b_i } \right\}$, $1 \le i \le n$. Also, for $x_i
,y_i  \in \left[ {a_i ,b_i } \right]$ and $t_i \in [0,1]$, define
$${\bf{t}} {\bf{x}}= \left( {t_1 x_1 ,t_2 x_2 , \ldots ,t_n x_n }
\right),$$ and
$$\left( {{\bf{1}} - {\bf{t}}} \right){\bf{y}} = \left( {\left( {1 - t_1 } \right)y_1 ,\left( {1 - t_2 } \right)y_2 , \ldots ,\left( {1 - t_n } \right)y_n } \right).$$

Let $f: \left[ {{\bf{a}},{\bf{b}}} \right]\subseteq \mathbb{R}^n
\to \mathbb{R}$, for the vector ${\bf{c}}$, we define
\begin{align*}
\sum\limits_{{\bf{c}}} {f\left( {{\bf{c}}} \right)} : =
\sum\limits_{\mathop {c_{i}   \in \left\{ {a_{i} ,b_{i} }
\right\}}\limits_{1 \le i\le n} } {f\left( {c_1 ,c_2 , \ldots ,c_n
} \right)},
\end{align*}
for all possible choices of $c_{i}  \in \left\{ {a_{i} ,b_{i} }
\right\}$, ($i=1,2,\cdots,n$).

\begin{definition}
A subset $\mathbb{D} \subseteq \mathbb{R}^n$ is said to be
${\bf{n}}$-fold convex if and only if whenever ${\bf{x}},{\bf{y}}
\in \mathbb{D}$ then ${\bf{t}} {\bf{x}}+ \left( {{\bf{1}} -
{\bf{t}}} \right){\bf{y}} \in \mathbb{D}$.
\end{definition}

\begin{corollary}
Every convex subset of $\mathbb{R}^n$ is an ${\bf{n}}$-fold
convex, and the converse is not true in general.
\end{corollary}

\begin{proof}
Follows directly from the definition.
\end{proof}

There is a subset $\mathbb{D} \subseteq \mathbb{R}^n$ which is
${\bf{n}}$-fold convex but is not convex. For example, consider
$\mathbb{D} \subset \mathbb{R}^2$, in the Figure (\ref{figer1}).
\begin{figure}[!h]
\includegraphics[angle=0,width=2in]{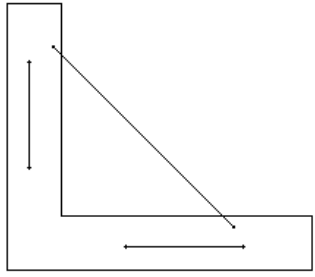}
\caption{${\bf{2}}$-fold convex set which is not convex.}
\label{figer1}
\end{figure}
On the other hand, there is a subset $\mathbb{D} \subseteq
\mathbb{R}^n$ which is not convex nor ${\bf{n}}$-fold, see the
Figure (\ref{figer2}).
\begin{figure}[!h]
\includegraphics[angle=0,width=2in]{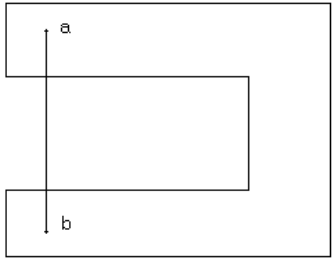}
\caption{A non convex set nor ${\bf{2}}$-fold convex.}
\label{figer2}
\end{figure}

\begin{definition}
A function $f:\left[ {{\bf{a}},{\bf{b}}} \right] \to \mathbb{R}$
is said to be ${\bf{n}}$-fold convex or convex on the coordinates
if and only if the inequality
\begin{align}
\label{eq2.1}f\left( {{\bf{t}} {\bf{x}}+ \left( {{\bf{1}} -
{\bf{t}}} \right){\bf{y}}} \right) \le \left( {\prod\limits_{1 \le
i \le n} {p_i } } \right)\sum\limits_{\bf{c}} {f\left( {\bf{c}}
\right)},
\end{align}
holds, for all ${\bf{x}}, {\bf{y}} \in  \left[ {{\bf{a}},{\bf{b}}}
\right]$ and ${\bf{t}} \in  \left[ {{\bf{0}},{\bf{1}}} \right]$,
where,
\begin{align}
\label{eq2.2}p_i  = \left\{ \begin{array}{l}
 t_i ,\,\,\,\,\,\,\,\,\,\,\,\,\,\,\,\,\text{if}\,\,\,\,\,\,\,c_i  = a_i  \\
 1 - t_i ,\,\,\,\,\,\,\text{if}\,\,\,\,\,\,\,c_i  = b_i  \\
 \end{array} \right.
\end{align}
for all $1\le i \le n$. Equivalently, $f$ is said to be
${\bf{n}}$-fold convex on $\left[ {{\bf{a}},{\bf{b}}} \right]$ iff
$f$ is convex in each coordinate separately on $[a_i,b_i]$ for all
$i=1,2,\cdots, n$. On the other hand, $f$ is called
${\bf{n}}$-fold concave if the inequality (\ref{eq2.1}) is
reversed.
\end{definition}

\begin{corollary}
Every  convex function defined on $\left[ {{\bf{a}},{\bf{b}}}
\right]\subseteq\mathbb{R}^n$ is ${\bf{n}}$-fold convex, and the
converse is not true in general.
\end{corollary}
\begin{proof}
Consider $f:\left[ {{\bf{a}},{\bf{b}}} \right] \to \mathbb{R}$ be
an ${\bf{n}}$-fold convex function. We carry out our proof using
induction.

Let
\begin{align}
P(n): f\left( {{\bf{t}} {\bf{x}}+ \left( {{\bf{1}} - {\bf{t}}}
\right){\bf{y}}} \right) \le \left( {\prod\limits_{1 \le i \le n}
{p_i } } \right)\sum\limits_{\bf{c}} {f\left( {\bf{c}}
\right)},\,\,\, n \in \mathbb{N}
\end{align}
holds, for all ${\bf{x}}, {\bf{y}} \in  \left[ {{\bf{a}},{\bf{b}}}
\right]$ and ${\bf{t}} \in  \left[ {{\bf{0}},{\bf{1}}} \right]$,
where,
\begin{align*}
\sum\limits_{\bf{c}} {f\left( {\bf{c}} \right)}=
\sum\limits_{\mathop {c_{i}   \in \left\{ {x_{i} ,y_{i} }
\right\}}\limits_{1 \le i\le n} } {f\left( {c_1 ,c_2 , \ldots ,c_n
} \right)},
\end{align*}
for all possible choices of $c_{i}  \in \left\{ {x_{i} ,y_{i}
}\right\}$, and
\begin{align*}
p_i  = \left\{ \begin{array}{l}
 t_i ,\,\,\,\,\,\,\,\,\,\,\,\,\,\,\,\,\text{if}\,\,\,\,\,\,\,c_i  = x_i  \\
 1 - t_i ,\,\,\,\,\,\,\text{if}\,\,\,\,\,\,\,c_i  = y_i  \\
 \end{array} \right.
\end{align*}
for all $1\le i \le n$.

For $n=2$, let $\left[ {{\bf{a}},{\bf{b}}} \right]=\left[
{a_1,b_1} \right]\times \left[ {a_2,b_2} \right]$, for each pair
${\bf{x}},{\bf{y}}\in \left[ {{\bf{a}},{\bf{b}}} \right]$;
${\bf{x}} = (x_1,x_2)$ and ${\bf{y}}=(y_1,y_2)$, since $f$ is
convex on $\left[ {{\bf{a}},{\bf{b}}} \right]$, then
\begin{align}
P(2): &f\left( {t_1 x_1  + \left( {1 - t_1 } \right)y_1 ,z}
\right)
\nonumber\\
&\le t_1 f\left( {x_1 ,z} \right) + \left( {1 - t_1 }
\right)f\left( {y_1 ,z} \right)\nonumber
 \\
&= t_1 f\left( {x_1 ,t_2 x_2  + \left( {1 - t_2 } \right)y_2 }
\right) + \left( {1 - t_1 } \right)f\left( {y_1 ,t_2 x_2  + \left(
{1 - t_2 } \right)y_2 } \right)\nonumber
  \\
&\le t_1 t_2 f\left( {x_1 ,x_2 } \right) + t_1 \left( {1 - t_2 }
\right)f\left( {x_1 ,y_2 } \right) + \left( {1 - t_1 } \right)t_2
f\left( {y_1 ,x_2 } \right)\nonumber
\\
&\qquad+ \left( {1 - t_1 } \right)\left( {1 - t_2 } \right)f\left(
{y_1 ,y_2 } \right)\nonumber
\\
&=\left( {\prod\limits_{i = 1}^2 {p_i } }
\right)\sum\limits_{\scriptstyle c_i  \in \left\{ {x_i ,y_i }
\right\} \hfill \atop
  \scriptstyle 1 \le i \le 2 \hfill} {f\left( {\bf{c}} \right)},
\end{align}
where ${\bf{c}}=(c_1,c_2)$, which mean that $f$ is ${\bf{2}}$-fold
convex on $\left[ {{\bf{a}},{\bf{b}}} \right]$.

For $n=k$, assume that $P(k)$ holds, and let $\left[
{{\bf{a}},{\bf{b}}} \right] =\prod\limits_{i = 1}^k {\left[ {a_i
,b_i } \right]} $, ${\bf{x}} = (x_1,x_2,\cdots,x_k)$ and
${\bf{y}}=(y_1,y_2,\cdots,y_k)$, since $f$ is ${\bf{k}}$-fold
convex on $\left[ {{\bf{a}},{\bf{b}}} \right]$, then
\begin{align}
f\left( {{\bf{t}}{\bf{x}} + \left( {1 - {\bf{t}}} \right){\bf{y}}}
\right) \le
 \left( {\prod\limits_{i = 1}^k {p_i } }
\right)\sum\limits_{\scriptstyle c_i  \in \left\{ {x_i ,y_i }
\right\} \hfill \atop
  \scriptstyle 1 \le i \le k \hfill} {f\left( {\bf{c}}
  \right)}\label{eq2.5}
\end{align}
for all $t\in [0,1]$ and ${\bf{x}},{\bf{y}} \in\left[
{{\bf{a}},{\bf{b}}} \right]$, where
\begin{align*}
p_i  = \left\{ \begin{array}{l}
 t_i ,\,\,\,\,\,\,\,\,\,\,\,\,\,\,\,\,\text{if}\,\,\,\,\,\,\,c_i  = x_i  \\
 1 - t_i ,\,\,\,\,\,\,\text{if}\,\,\,\,\,\,\,c_i  = y_i  \\
 \end{array} \right.
\end{align*}
for all $i=1,2,\cdots, k$.

It remains to show that $P(n)$ holds when $n=k+1$, therefore
\begin{align*}
P(k+1)&: f\left( {t_1 x_1  + \left( {1 - t_1 } \right)y_1 , \ldots
,t_{k + 1} x_{k + 1}  + \left( {1 - t_{k + 1} } \right)y_{k + 1} }
\right)
\\
&= f\left( {t_1 x_1  + \left( {1 - t_1 } \right)y_1 , \ldots ,t_k
x_k  + \left( {1 - t_k } \right)y_k ,t_{k + 1} x_{k + 1}  + \left(
{1 - t_{k + 1} } \right)y_{k + 1} } \right)
  \\
&= \left( {\prod\limits_{i = 1}^k {p_i } }
\right)\sum\limits_{\scriptstyle c_i  \in \left\{ {x_i ,y_i }
\right\} \hfill \atop
  \scriptstyle 1 \le i \le k \hfill} {f\left( {{\bf{c}},t_{k + 1} x_{k + 1}  + \left(
{1 - t_{k + 1} } \right)y_{k + 1}} \right)}
\\
&\le t_{k+1} \left( {\prod\limits_{i = 1}^k {p_i } }
\right)\sum\limits_{\scriptstyle c_i  \in \left\{ {x_i ,y_i }
\right\} \hfill \atop
  \scriptstyle 1 \le i \le k \hfill} {f\left( {{\bf{c}},x_{k + 1}} \right)}
\\
&\qquad+\left( {1 - t_{k + 1} } \right) \left( {\prod\limits_{i =
1}^k {p_i } } \right)\sum\limits_{\scriptstyle c_i  \in \left\{
{x_i ,y_i } \right\} \hfill \atop
  \scriptstyle 1 \le i \le k \hfill} {f\left( {{\bf{c}},y_{k + 1}} \right)}
,\,\,\,\,\,\,\,\text{(follows from (\ref{eq2.5}))}
  \\
&= \left( {\prod\limits_{i = 1}^{k + 1} {p_i } }
\right)\sum\limits_{\scriptstyle c_i  \in \left\{ {x_i ,y_i }
\right\} \hfill \atop
  \scriptstyle 1 \le i \le k+1 \hfill} {f\left( {\bf{c}} \right)}
\end{align*}
where,
\begin{align*}
p_i  = \left\{ \begin{array}{l}
 t_i ,\,\,\,\,\,\,\,\,\,\,\,\,\,\,\,\,\text{if}\,\,\,\,\,\,\,c_i  = x_i  \\
 1 - t_i ,\,\,\,\,\,\,\text{if}\,\,\,\,\,\,\,c_i  = y_i  \\
 \end{array} \right.
\end{align*}
for all $i=1,2,\cdots, k+1$. Hence, by mathematical induction,
$P(n)$ holds for all $n \in \mathrm{}\mathbb{N}$. On the other
hand, the function $f:[0,1]^2 \to [0,\infty)$, $f(x,y)=y$ is
${\bf{2}}$-fold convex on $[0,1]^2 $ but is not convex. The
reverse of (\ref{eq2.1}) follows directly by replacing $f$ by
$-f$, and thus the proof is completely established.
\end{proof}

The following Jensen's type inequality holds:
\begin{theorem}
\label{thm1}Let $f:\left[ {{\bf{a}},{\bf{b}}} \right] \to
\mathbb{R}$ be ${\bf{k}}$-fold convex. Let $x_i^{\left( r
\right)}$ be a finite sequence of real numbers, for all
$i,r=1,2,\cdots, k$, and consider $ {\bf{x}} = \left( {x_i^{\left(
1 \right)} ,x_i^{\left( 2 \right)} , \ldots ,x_i^{\left( k
\right)} } \right)$, $\alpha  = \left( {\alpha _i^{\left( 1
\right)} ,\alpha _i^{\left( 2 \right)} , \ldots ,\alpha _i^{\left(
k \right)} } \right)$, with $\sum \alpha   = {\bf{1}}$, i.e., $
\sum\limits_{i = 1}^k {\alpha _i^{\left( r \right)} }  = 1$, for
all $r=1,2,\cdots,k$.
 Then the inequality
\begin{align}
\label{eq.jensen}f\left( {\sum {\alpha {\bf{x}}} } \right)
\le\left( {\prod\limits_{1 \le i \le k} {\gamma _i } } \right)
\cdot \sum\limits_{\bf{c}} { f\left( {\bf{c}} \right)}
\end{align}
holds,  where
\begin{align*}
\sum {\alpha {\bf{x}}} : = \left( {\sum\limits_{i = 1}^{k} {\alpha
_i^{\left( 1 \right)} x_i^{\left( 1 \right)} } ,\sum\limits_{i =
1}^{k} {\alpha _i^{\left( 2 \right)} x_i^{\left( 2 \right)} } ,
\ldots ,\sum\limits_{i = 1}^{k} {\alpha _i^{\left( k \right)}
x_i^{\left( k \right)} } } \right),
\end{align*}
${\bf{c}}: = \left( {c_1 ,c_2, \ldots ,c_k } \right),c_i  \in
\left\{ {x_i^{\left( j \right)} } \right\}_{j = 1}^k$ and
\begin{align*}
\gamma _i  = \left\{ \begin{array}{l}
 \alpha _i^{\left( 1 \right)} ,\,\,\,\,\,\,\,\,\,\,\,c_i  = x_i^{\left( 1 \right)}  \\
 \alpha _i^{\left( 2 \right)} ,\,\,\,\,\,\,\,\,\,\,\,c_i  = x_i^{\left( 2 \right)}  \\
  \vdots  \\
 \alpha _i^{\left( k \right)} ,\,\,\,\,\,\,\,\,\,\,\,c_i  = x_i^{\left( k \right)}  \\
 \end{array} \right.
\end{align*}
If $f$ is ${\bf{n}}$-fold concave then the inequality
(\ref{eq.jensen}) is reversed.
\end{theorem}

\begin{proof}
Use the definition of ${\bf{n}}$-fold convex and apply the
classical Jensen's inequality for convex function of one variable
in each variable.
\end{proof}

The following Hermite-Hadamard inequality holds:
\begin{theorem}
\label{thm2}Let $f:\left[ {{\bf{a}},{\bf{b}}} \right] \to
\mathbb{R}$ be ${\bf{n}}$-fold convex. Then the inequality
\begin{align}
\label{eq.HH}f\left( {\frac{{{\bf{a}} + {\bf{b}}}}{2}} \right) \le
\frac{1}{{{\bf{b}} - {\bf{a}}}}\int_{\bf{a}}^{\bf{b}} {f\left(
{\bf{x}} \right)d{\bf{x}}}  \le \frac{1}{{2^n
}}\sum\limits_{\bf{c}} {f\left( {\bf{c}} \right)},
\end{align}
holds, where
\begin{align*}
\sum\limits_{\bf{c}} {f\left( {\bf{c}} \right)} : =
\sum\limits_{\mathop {c_i  \in \left\{ {a_i ,b_i }
\right\}}\limits_{1 \le i \le n} } {f\left( {c_1, c_2, \ldots ,c_n
} \right)}.
\end{align*}
The inequality (\ref{eq.HH}) is sharp. If $f$ is ${\bf{n}}$-fold
concave then the inequality (\ref{eq.HH}) is reversed.
\end{theorem}

\begin{proof}
Since $f$ is ${\bf{n}}$-fold convex on $\left[ {{\bf{a}},{\bf{b}}}
\right]$, then for all ${\bf{t}} \in \left[{{\bf{0}},{\bf{1}}}
\right]$, we have
\begin{align}
\label{eq2.8}f\left( {{\bf{ta}} + \left( {{\bf{1}} - {\bf{t}}}
\right){\bf{b}}} \right) \le \sum\limits_{\bf{c}} {\left(
{\prod\limits_{1 \le i \le n} {p_i } } \right)f\left( {\bf{c}}
\right)}.
\end{align}
Integrating (\ref{eq2.8}) with respect to ${\bf{t}}$ on $\left[
{{\bf{0}},{\bf{1}}} \right]$ we get
\begin{align}
\int_{\bf{0}}^{\bf{1}} {f\left( {{\bf{ta}} + \left( {{\bf{1}} -
{\bf{t}}} \right){\bf{b}}} \right)d{\bf{t}}}  &\le
\int_{\bf{0}}^{\bf{1}} {\left( {\sum\limits_{\bf{c}} {\left(
{\prod\limits_{1 \le i \le n} {p_i } } \right)f\left( {\bf{c}}
\right)} } \right)d{\bf{t}}}
\nonumber\\
&=\left( {\sum\limits_{\bf{c}} {f\left( {\bf{c}} \right)} }
\right)\int_{\bf{0}}^{\bf{1}} {\left( {\prod\limits_{1 \le i \le
n} {p_i } } \right)d{\bf{t}}}
\nonumber\\
&= \frac{1}{{2^n }}\sum\limits_{\bf{c}} {f\left( {\bf{c}}
\right)}\label{eq2.9}
\end{align}
where, $p_i$ is defined in (\ref{eq2.2}).

On the other hand, again since $f$ is ${\bf{n}}$-fold convex on
$\left[ {{\bf{a}},{\bf{b}}} \right]$, then for ${\bf {t}} \in
\left[ {{\bf{0}},{\bf{1}}} \right]$, we have
\begin{align}
f\left( {\frac{{{\bf{a}} + {\bf{b}}}}{2}} \right) &= f\left(
{\frac{{{\bf{ta}} + \left( {{\bf{1}} - {\bf{t}}}
\right){\bf{b}}}}{2} + \frac{{{\bf{tb}} + \left( {{\bf{1}} -
{\bf{t}}} \right){\bf{a}}}}{2}} \right)
\nonumber\\
&\le \frac{1}{2}\left[ {f\left( {{\bf{ta}} + \left( {{\bf{1}} -
{\bf{t}}} \right){\bf{b}}} \right) + f\left( {{\bf{tb}} + \left(
{{\bf{1}} - {\bf{t}}} \right){\bf{a}}} \right)} \right].
\label{eq2.10}
\end{align}
Integrating inequality (\ref{eq2.10}) with respect to ${\bf {t}}$
on $\left[ {{\bf{0}},{\bf{1}}} \right]$ we get
\begin{align}
f\left( {\frac{{{\bf{a}} + {\bf{b}}}}{2}} \right) &\le
\frac{1}{2}\int_{\bf{0}}^{\bf{1}} {\left[ {f\left( {{\bf{ta}} +
\left( {{\bf{1}} - {\bf{t}}} \right){\bf{b}}} \right) + f\left(
{{\bf{tb}} + \left( {{\bf{1}} - {\bf{t}}} \right){\bf{a}}}
\right)} \right]d{\bf{t}}}
\nonumber\\
&= \frac{1}{2}\int_{\bf{0}}^{\bf{1}} {f\left( {{\bf{ta}} + \left(
{{\bf{1}} - {\bf{t}}} \right){\bf{b}}} \right)d{\bf{t}}}  +
\frac{1}{2}\int_{\bf{0}}^{\bf{1}} {f\left( {{\bf{tb}} + \left(
{{\bf{1}} - {\bf{t}}} \right){\bf{a}}} \right)d{\bf{t}}}.
\label{eq2.11}
\end{align}
By putting ${\bf{1}} - {\bf{t}} = {\bf{s}}$, in the second
integral on the right-hand side of (\ref{eq2.11}), we get
\begin{align}
f\left( {\frac{{{\bf{a}} + {\bf{b}}}}{2}} \right) &\le
\frac{1}{2}\int_{\bf{0}}^{\bf{1}} {f\left( {{\bf{ta}} + \left(
{{\bf{1}} - {\bf{t}}} \right){\bf{b}}} \right)d{\bf{t}}}  +
\frac{1}{2}\int_{\bf{0}}^{\bf{1}} {f\left( {{\bf{tb}} + \left(
{{\bf{1}} - {\bf{t}}} \right){\bf{a}}} \right)d{\bf{t}}}
\nonumber\\
&=  \int_{\bf{0}}^{\bf{1}} {f\left( {{\bf{ta}} + \left( {{\bf{1}}
- {\bf{t}}} \right){\bf{b}}} \right)d{\bf{t}}}. \label{eq2.12}
\end{align}
From (\ref{eq2.9}) and (\ref{eq2.12}), we get
\begin{align}
\label{eq2.13}f\left( {\frac{{{\bf{a}} + {\bf{b}}}}{2}} \right)
\le \int_{\bf{0}}^{\bf{1}} {f\left( {{\bf{ta}} + \left( {{\bf{1}}
- {\bf{t}}} \right){\bf{b}}} \right)d{\bf{t}}} \le \frac{1}{{2^n
}}\sum\limits_{\bf{c}} {f\left( {\bf{c}} \right)}.
\end{align}
By putting ${\bf{ta}} + \left( {{\bf{1}} - {\bf{t}}}
\right){\bf{b}} = {\bf{x}}$ in the integral involved in
(\ref{eq2.13}), it is easy to observe that
\begin{align}
\label{eq2.14}\int_{\bf{0}}^{\bf{1}} {f\left( {{\bf{ta}} + \left(
{{\bf{1}} - {\bf{t}}} \right){\bf{b}}} \right)d{\bf{t}}}  =
\frac{1}{{{\bf{b}} - {\bf{a}}}}\int_{\bf{a}}^{\bf{b}} {f\left(
{\bf{x}} \right)d{\bf{x}}}.
\end{align}
which proves the inequality (\ref{eq.HH}). The sharpness follows
by taking the function $f\left({{\bf{x}}}\right)= \prod\limits_{i
= 1, \ldots ,n} {x_i } $. If $f$ is ${\bf{n}}$-fold concave,
replacing $-f$ instead of $f$ in (\ref{eq.HH}) we get the required
result.
\end{proof}

Next, we consider a weighted version of (\ref{eq.HH}) which is
known as Fej\'{e}r inequality, before that we need the following
preliminary lemma:
\begin{lemma}\label{lemma}
Let $ f:\left[ {{\bf{a}},{\bf{b}}} \right]  \to \mathbb{R} $ be
${\bf{n}}$-fold convex function. Let ${\bf{x}}_{\bf{1}} = \left(
{x_1^{\left( 1 \right)} , \ldots ,x_1^{\left( n \right)} }
\right),{\bf{x}}_2 = \left( {x_2^{\left( 1 \right)} , \ldots
,x_2^{\left( n \right)} } \right) , {\bf{y}}_{\bf{1}} = \left(
{y_1^{\left( 1 \right)} , \ldots ,y_1^{\left( n \right)} }
\right),{\bf{y}}_{\bf{2}}= \left( {y_2^{\left( 1 \right)} , \ldots
,y_2^{\left( n \right)} } \right) $ be any vectors in $\left[
{{\bf{a}},{\bf{b}}} \right] $ such that $ {\bf{a}} \le {\bf{y_1}}
\le {\bf{x_1}} \le {\bf{x_2}} \le {\bf{y_2}}  \le {\bf{b}} $, with
${\bf{x}}_{\bf{1}} + {\bf{x}}_{\bf{2}}  = {\bf{y}}_{\bf{1}}  +
{\bf{y}}_{\bf{2}}$. Then, for the convex partial mappings $ f_i
:\left[ {a_i,b_i} \right] \to \mathbb{R}$, $ f_i \left( {t_i}
 \right) =
f\left( {z_1 , \ldots ,z_{i - 1} ,t_i ,z_{i + 1} , \ldots ,z_n }
\right) $,  for all fixed $ z_j \in \left[ {a_j,b_j} \right] $
$(j=1,2,\cdots,n)$ with $j\ne i$.  the following inequality holds:
\begin{multline}
\label{eq2.15}f\left( {z_1 , \ldots ,z_{i - 1} ,x_1^{\left( i
\right)} ,z_{i + 1} , \ldots ,z_n } \right) + f\left( {z_1 ,
\ldots ,z_{i - 1} ,x_2^{\left( i \right)} ,z_{i + 1} , \ldots ,z_n
} \right)
\\
\le f\left( {z_1 , \ldots ,z_{i - 1} ,y_1^{\left( i \right)} ,z_{i
+ 1} , \ldots ,z_n } \right) + f\left( {z_1 , \ldots ,z_{i - 1}
,y_2^{\left( i \right)} ,z_{i + 1} , \ldots ,z_n } \right)
\end{multline}
\end{lemma}
\begin{proof}
Consider $ f_i :\left[ {a_i,b_i} \right] \to \mathbb{R} $, $ f_i
\left( {t_i}
 \right) =
f\left( {z_1 , \ldots ,z_{i - 1} ,t_i ,z_{i + 1} , \ldots ,z_n }
\right) $,  for all fixed $ z_j \in \left[ {a_j,b_j} \right] $
$(j=1,2,\cdots,n)$ with $j\ne i$. If ${\bf{y}}_1 = {\bf{y}}_2$
then we are done. Suppose ${\bf{y}}_1 \ne {\bf{y}}_2$ and write
$$x_1^{(i)}  = \frac{{y_2^{(i)}  - x_1^{(i)} }}{{y_2^{(i)}  -
y_1^{(i)} }}y_1^{(i)}  + \frac{{x_1^{(i)} - y_1^{(i)}
}}{{y_2^{(i)}  - y_1^{(i)} }}y_2^{(i)},$$and$$x_2^{(i)}  =
\frac{{y_2^{(i)}  - x_2^{(i)} }}{{y_2^{(i)} - y_1^{(i)}
}}y_1^{(i)}  + \frac{{x_2^{(i)}  - y_1^{(i)} }}{{y_2^{(i)} -
y_1^{(i)} }}y_2^{(i)},$$ for all $i=1,2,\cdots,n$.

Since $f_i$ is convex on $\left[ {a_i,b_i} \right] $, we have
\begin{align*}
&f\left( {z_1 , \ldots ,z_{i - 1} ,x_1^{\left( i \right)} ,z_{i +
1} , \ldots ,z_n } \right) + f\left( {z_1 , \ldots ,z_{i - 1}
,x_2^{\left( i \right)} ,z_{i + 1} , \ldots ,z_n } \right)
\\
&\le \frac{{y_2^{(i)}  - x_1^{(i)} }}{{y_2^{(i)}  - y_1^{(i)}
}}f\left( {z_1 , \ldots ,z_{i - 1} ,y_1^{\left( i \right)} ,z_{i +
1} , \ldots ,z_n } \right)
\\
&\qquad\qquad+ \frac{{x_1^{(i)} - y_1^{(i)} }}{{y_2^{(i)}  -
y_1^{(i)} }}f\left( {z_1 , \ldots ,z_{i - 1} ,y_2^{\left( i
\right)} ,z_{i + 1} , \ldots ,z_n } \right)
\\
&\qquad+ \frac{{y_2^{(i)}  - x_2^{(i)} }}{{y_2^{(i)} - y_1^{(i)}
}}f\left( {z_1 , \ldots ,z_{i - 1} ,y_1^{\left( i \right)} ,z_{i +
1} , \ldots ,z_n } \right)
\\
&\qquad\qquad+ \frac{{x_2^{(i)}  - y_1^{(i)} }}{{y_2^{(i)} -
y_1^{(i)} }}f\left( {z_1 , \ldots ,z_{i - 1} ,y_2^{\left( i
\right)} ,z_{i + 1} , \ldots ,z_n } \right)
\\
&= \frac{{2y_2^{(i)}   - \left( {x_1^{(i)}   + x_2^{(i)}  }
\right)}}{{y_2^{(i)}   - y_1^{(i)}  }}f\left( {z_1 , \ldots ,z_{i
- 1} ,y_1^{\left( i \right)} ,z_{i + 1} , \ldots ,z_n } \right)
\\
&\qquad+ \frac{{\left( {x_1^{(i)}   + x_2^{(i)}  } \right) -
2y_1^{(i)} }}{{y_2^{(i)}  - y_1^{(i)}  }}f\left( {z_1 , \ldots
,z_{i - 1} ,y_2^{\left( i \right)} ,z_{i + 1} , \ldots ,z_n }
\right)
\\
&= f\left( {z_1 , \ldots ,z_{i - 1} ,y_1^{\left( i \right)} ,z_{i
+ 1} , \ldots ,z_n } \right)+f\left( {z_1 , \ldots ,z_{i - 1}
,y_2^{\left( i \right)} ,z_{i + 1} , \ldots ,z_n } \right),
\end{align*}
for all $i=1,2,\cdots,n$, which shows that (\ref{eq2.15}) holds.
\end{proof}

A Fej\v{e}r type inequality may be stated as follows:
\begin{theorem}
\label{thm3} Let $f:\left[ {{\bf{a}},{\bf{b}}} \right] \to
\mathbb{R}$ be ${\bf{n}}$-fold convex. Then the double inequality
\begin{align}
\label{fejer}f\left( {\frac{{{\bf{a}} + {\bf{b}}}}{2}} \right)\le
\frac{\int_{\bf{a}}^{\bf{b}} {p\left( {\bf{x}} \right)f\left(
{\bf{x}} \right)d{\bf{x}}}}{\int_{\bf{a}}^{\bf{b}} {p\left(
{\bf{x}} \right)d{\bf{x}}}} \le \frac{1}{{2^n
}}\sum\limits_{\bf{c}} {f\left( {\bf{c}} \right)}
\end{align}
holds, where $p:\left[ {{\bf{a}},{\bf{b}}} \right]\to \mathbb{R}$
is positive, integrable, and symmetric about $x_i = \frac{{a_i +
b_i}}{2}$ for all $i=1,2,\cdots, n$. The above inequalities are
sharp.
\end{theorem}

\begin{proof}
Since $p$ is positive, integrable, and symmetric about $x_i =
\frac{{a_i + b_i}}{2}$ for all $i=1,2,\cdots, n$. Then, by Lemma
\ref{lemma} one has:
\begin{align*}
 f\left( {\frac{{{\bf{a}} + {\bf{b}}}}{2}}
\right)\int\limits_{\bf{a}}^{\bf{b}}
{p\left({{\bf{x}}}\right)d{\bf{x}}} &=
\int_{\bf{a}}^{{\textstyle{{{\bf{a}} + {\bf{b}}} \over 2}}}
{f\left( {\frac{{{\bf{a}} + {\bf{b}}}}{2}} \right)p\left( {\bf{x}}
\right)d{\bf{x}}}
\\
&\qquad+ \int_{\bf{a}}^{{\textstyle{{{\bf{a}} + {\bf{b}}} \over
2}}} {f\left( {\frac{{{\bf{a}} + {\bf{b}}}}{2}} \right)p\left(
{{\bf{a}} + {\bf{b}} - {\bf{x}}} \right)d{\bf{x}}}
\\
&=\int_{\bf{a}}^{{\textstyle{{{\bf{a}} + {\bf{b}}} \over 2}}}
{\left[ {f\left( {\frac{{{\bf{a}} + {\bf{b}}}}{2}} \right) +
f\left( {\frac{{{\bf{a}} + {\bf{b}}}}{2}} \right)} \right]p\left(
{\bf{x}} \right)d{\bf{x}}}
\\
&\le \int_{\bf{a}}^{{\textstyle{{{\bf{a}} + {\bf{b}}} \over 2}}}
{\left[ {f\left( {{\bf{x}}} \right) + f\left( {{\bf{a}} + {\bf{b}}
- {\bf{x}}} \right)} \right]p\left( {\bf{x}} \right)d{\bf{x}}}
\\
&=  \int_{\bf{a}}^{{\textstyle{{{\bf{a}} + {\bf{b}}} \over 2}}}
{f\left( {{\bf{x}}} \right)p\left( {\bf{x}} \right)d{\bf{x}}}  +
\int_{{\textstyle{{{\bf{a}} + {\bf{b}}} \over 2}}}^{\bf{b}}
{f\left( {{\bf{x}}} \right)p\left( {\bf{x}} \right)d{\bf{x}}}
\\
&=\int_{\bf{a}}^{\bf{b}} {p\left( {\bf{x}} \right)f\left( {\bf{x}}
\right)d{\bf{x}}}
\end{align*}
and
\begin{align*}
& \frac{1}{{2^n }}\sum\limits_{\bf{c}} {f\left( {\bf{c}} \right)}
\int_{\bf{a}}^{\bf{b}} {p\left( {\bf{x}} \right)d{\bf{x}}}
\\
&=  \int_{\bf{a}}^{{\textstyle{{{\bf{a}} + {\bf{b}}} \over 2}}}
{{\left[ {\frac{1}{{2^n }}\sum\limits_{\bf{c}} {f\left( {\bf{c}}
\right)}} \right]p\left( {\bf{x}}
\right)}d{\bf{x}}}+\int_{{{\textstyle{{{\bf{a}} + {\bf{b}}} \over
2}}}}^{\bf{b}} {{\left[ {\frac{1}{{2^n }}\sum\limits_{\bf{c}}
{f\left( {\bf{c}} \right)}} \right]p\left(
{\bf{{\bf{a}}+{\bf{b}}-{\bf{x}}}} \right)}d{\bf{x}}}
\\
&=  \int_{\bf{a}}^{{\textstyle{{{\bf{a}} + {\bf{b}}} \over 2}}}
{{\left[ {\frac{1}{{2^n }}\sum\limits_{\bf{c}} {f\left( {\bf{c}}
\right)}} \right]p\left( {\bf{x}} \right)}d{\bf{x}}}
\\
&\ge\int_{\bf{a}}^{{\textstyle{{{\bf{a}} + {\bf{b}}} \over 2}}}
{\left[ {f\left( {\bf{x}} \right) + f\left( {{\bf{a}} + {\bf{b}} -
{\bf{x}}} \right)} \right]p\left( {\bf{x}} \right)d{\bf{x}}}
\\
&=  \int_{\bf{a}}^{{\textstyle{{{\bf{a}} + {\bf{b}}} \over 2}}}
{{p\left( {\bf{x}} \right)f\left( {\bf{x}}
\right)}d{\bf{x}}}+\int_{{{\textstyle{{{\bf{a}} + {\bf{b}}} \over
2}}}}^{\bf{b}} {{p\left( {{\bf{x}}} \right)f\left( {\bf{x}}
\right)}d{\bf{x}}}
\\
&= \int_{\bf{a}}^{\bf{b}} {p\left( {\bf{x}} \right)f\left(
{\bf{x}} \right)d{\bf{x}}},
\end{align*}
which proves (\ref{fejer}). To prove the sharpness in
(\ref{fejer}), take $p({\bf{x}}) = 1$, then the inequality
(\ref{fejer}) is reduced to the double inequality (\ref{eq.HH}),
and therefore if we choose $f({\b{x}}) =  \prod\limits_{i = 1,
\ldots ,n} {x_i } $, in (\ref{fejer}), then the equality holds,
which shows that (\ref{fejer}) is sharp, and thus the proof is
completely finished.
\end{proof}

\section{A Matrix version of H.--H. Inequality}

A matrix function, $f(A)$, or function of a matrix can have
several different meanings. It can be an operation on a matrix
producing a scalar, such as $tra{(A)}$ and $\det{(A)}$; it can be
a mapping from a matrix space to a matrix space, like $f(A) =
A^2$; it can also be entrywise operations on the matrix, for
instance, $g(A) = (a_{ij})^2$.

A natural generalization of the classical Hermite--Hadamard
inequality (\ref{HH}) to Hermitian matrices could be the double
inequality
\begin{align}
\label{eq3.1}f\left( {\frac{{A + B}}{2}} \right) \le \int_0^1
{f\left( {tA + \left( {1 - t} \right)B} \right)dt}  \le
\frac{{f\left( A \right) + f\left( B \right)}}{2}
\end{align}
which is however not true, in general as shown recently in
\cite{Moslehian}.

Moslehian \cite{Moslehian} introduced several matrix and operator
inequalities of Hermite–-Hadamard type and he presented some
operator inequalities of Hermite-Hadamard type in which the
classical convexity was used instead of the operator convexity.\\

In this section, we introduce a matrix version of
Hermite--Hadamard inequality for function of a matrix producing a
scalar.

Let $\mathcal{M}_{n\times n} (\mathbb{R})$ be the set of all real
$(n \times n)$--matrices with real entries, given a function
$f:\mathcal{M}_{n\times n} (\mathbb{R})\to \mathbb{R}$ and $A,B\in
\mathcal{M}_{n\times n} (\mathbb{R})$. Clearly, each square
$n$-matrix is just a point in $\mathbb{R}^{n^2}$ . For example a
$2 \times 2$-matrix is just a point in $\mathbb{R}^4$; i.e., it
has four real coordinates; e.g., the matrix $\left(
{\begin{array}{*{20}c}
   1 & 2  \\
   3 & 4  \\
\end{array}} \right)$ is just the vector $(1, 2, 3, 4)$. At first this may seem an oversimplification because it ignores the matrix
product. Thus we define $\mathcal{M}_{2\times 2}(\mathbb{R})$ to
be in $\mathbb{R}^4$ with the following product defined in it
\begin{align*}
\left( {a,b,c,d} \right)\left( {u,v,x,y} \right) = \left( {au +
bx,av + by{\rm{,}}cu + dx,cv + dy} \right)
\end{align*}
which is just the matrix product
\begin{align*}
\left( {\begin{array}{*{20}c}
   a & b  \\
   c & d  \\
\end{array}} \right)\left( {\begin{array}{*{20}c}
   u & v  \\
   x & y  \\
\end{array}} \right)
\end{align*}
written as a vector in $\mathbb{R}^4$. Finally, the integration
limits $A,B$ are just vectors in $\mathbb{R}^{n^2}$ (with n = 2 in
our case). Thus the integral is really multiple integral.

To state our result we need to understand the following
terminologies:
\begin{align*}
X &= \left( {x_{ij} } \right)_{n \times n} = \left( {x_{11} ,
\ldots ,x_{1n} ,x_{21} , \ldots ,x_{2n} , \ldots ,x_{n1} , \ldots
,x_{nn} } \right)
\\
A&\le B  \Leftrightarrow a_{ij}\le b_{ij}, \forall i,j=1,\cdots,
n.
\end{align*}
Given a matrix function of real variables $F:\mathcal{M}_{n\times
n} (\mathbb{R})\to \mathbb{R}$. For a matrix $C =
(c_{ij})_{n\times n}$,  we define
\begin{align*} \sum\limits_{C} {F\left( {C} \right)} : =
\sum\limits_{\mathop {c_{ij}   \in \left\{ {a_{ij} ,b_{ij} }
\right\}}\limits_{1 \le i,j \le n} } {F\left( { \left( {c_{ij} }
\right)_{n \times n} } \right)}.
\end{align*}
for all possible choices of $c_{ij}  \in \left\{ {a_{ij} ,b_{ij} }
\right\}$.

We define the matrix-interval to be  $\left[ {A,B} \right] =
\prod\limits_{j = 1}^n {\prod\limits_{i = 1}^n {\left[ {a_{ij}
,b_{ij} } \right]} }$, with length to be $B - A = \prod\limits_{j
= 1}^n {\prod\limits_{i = 1}^n {\left( {b_{ij} - a_{ij} } \right)}
} $. Depending on this, we understand  $\int_A^B {f\left( X
\right)dX}$ to be:
\begin{align*}
\int_A^B {f\left( X \right)dX}  = \int_{a_{nn} }^{b_{nn} } {
\cdots \int_{a_{11} }^{b_{11} } {f\left( {x_{11} , \ldots ,x_{nn}
} \right)dx_{11}  \cdots dx_{nn} } }.
\end{align*}

Next result illustrate a matrix version of H.--H. inequality for
function of a matrix producing a scalar:
\begin{theorem}
\label{thm4}Let $A,B\in \mathcal{M}_{n\times n} (\mathbb{R})$ with
$A < B$. Let $f:\left[ {A,B} \right] \to \mathbb{R}$ be
${\bf{n}}^2$-fold convex. Then the inequality
\begin{align}
\label{eq4.2}f\left( {\frac{{A + B}}{2}} \right) \le \frac{1}{{B -
A}}\int_A^B {f\left( X \right)dX}  \le \frac{1}{2^{n^2}}
\sum\limits_{\scriptstyle C = \left( {c_{ij} } \right)_{n \times
n}  \hfill \atop
  {\scriptstyle c_{ij}  \in \left\{ {a_{ij} ,b_{ij} } \right\} \hfill \atop
  \scriptstyle 1 \le i,j \le n \hfill}} {f\left( C \right)}.
\end{align}
holds, where
\begin{align*}
\sum\limits_{C} {f\left( {C} \right)} : =
\sum\limits_{\scriptstyle C = \left( {c_{ij} } \right)_{n \times
n}  \hfill \atop
  {\scriptstyle c_{ij}  \in \left\{ {a_{ij} ,b_{ij} } \right\} \hfill \atop
  \scriptstyle 1 \le i,j \le n \hfill}} {f\left( {c_{11},c_{12},\cdots,,c_{nn} } \right)}.
\end{align*}
The inequality (\ref{eq4.2}) is sharp. If $f$ is ${\bf{n}}^2$-fold
concave then the inequality (\ref{eq4.2}) is reversed.
\end{theorem}

\begin{proof}
The proof follows directly from Theorem \ref{thm2}.
\end{proof}

\begin{remark}
A Jensen's type inequality for matrix functions used above; may be
deduced in a similar manner as in Theorem {\ref{thm1}}.

\end{remark}

\bibliographystyle{amsplain}

\end{document}